\documentclass[11pt]{amsart}

\usepackage{amsmath}
\usepackage{amsfonts}
\usepackage{amssymb}
\usepackage{amscd}
\usepackage{amsthm}
\usepackage{framed}
\usepackage{fullpage}
\usepackage{graphicx}
\usepackage{latexsym}
\usepackage[numbers]{natbib}

\usepackage{hyperref}
\hypersetup{colorlinks,linkcolor={red},citecolor={blue},urlcolor={blue}}

\evensidemargin\oddsidemargin

\newtheorem{theorem}{Theorem}[section]
\newtheorem{lemma}[theorem]{Lemma}

\newtheorem{proposition}[theorem]{Proposition}

\theoremstyle{definition}
\newtheorem{definition}[theorem]{Definition}

\newtheorem{remark}[theorem]{Remark}

\newtheorem{conjecture}[theorem]{Conjecture}
\newtheorem{Algorithm}[theorem]{Algorithm}

\numberwithin{equation}{section}

\newcommand{\CC}{\mathbb C}
\newcommand{\HH}{\mathbb H}
\newcommand{\NN}{\mathbb N}

\newcommand{\RR}{\mathbb R}
\newcommand{\ZZ}{\mathbb Z}
\newcommand{\cO}{\mathcal O}
\newcommand{\SL}{\mathop{\mathrm {SL}}\nolimits}
\newcommand{\orb}{\operatorname{orb}}
\newcommand{\w}{\operatorname{w}}

\begin{document}

\title[]{Weyl invariant $E_8$ Jacobi forms and $E$-strings}

\author{Kaiwen Sun}

\address{Korea Institute for Advanced Study, 85 Hoegiro, Dongdaemun-gu, Seoul, South Korea}

\email{ksun@kias.re.kr}

\author{Haowu Wang}

\address{Center for Geometry and Physics, Institute for Basic Science (IBS), Pohang 37673, Korea}

\email{haowu.wangmath@gmail.com}


\date{\today}


\begin{abstract}
In 1992 Wirthm\"{u}ller showed that for any irreducible root system not of type $E_8$ the ring of weak Jacobi forms invariant under Weyl group is a polynomial algebra. However, it has recently been proved that for $E_8$ the ring is not a polynomial algebra. Weyl invariant $E_8$ Jacobi forms have many applications in string theory and it is an open problem to describe such forms. The scaled refined free energies of $E$-strings with certain $\eta$-function factors are conjectured to be Weyl invariant $E_8$ quasi holomorphic Jacobi forms. It is further observed that the scaled refined free energies up to some powers of $E_4$ can be written as  polynomials in nine Sakai's $E_8$ Jacobi forms and Eisenstein series $E_2$, $E_4$, $E_6$. Motivated by the physical conjectures, we prove that for any Weyl invariant $E_8$ Jacobi form $\phi_t$ of index $t$ the function $E_4^{[t/5]}\Delta^{[5t/6]}\phi_t$ can be expressed uniquely as a polynomial in $E_4$, $E_6$ and Sakai's forms, where $[x]$ is the integer part of $x$. This means that a Weyl invariant $E_8$ Jacobi form is completely determined by a solution of some linear equations. By solving the linear systems, we determine the generators of the free module of Weyl invariant $E_8$ weak (resp. holomorphic) Jacobi forms of given index $t$ when $t\leq 13$ (resp. $t\leq 11$). 
\end{abstract}

\maketitle

\section{introduction}
In 1985 Eichler and Zagier introduced the theory of Jacobi forms in their monograph \cite{EZ85}. Jacobi forms are an elegant intermediate between different types of modular forms and have many applications in mathematics and physics. In 1992 Wirthm\"{u}ller \cite{Wir92} investigated Weyl invariant Jacobi forms associated with root systems. Let $R$ be an irreducible root system of rank $r$. A $W(R)$-invariant Jacobi form is a holomorphic function of complex variables $\tau \in \HH$ and $\mathfrak{z}\in R\otimes \CC$ which is modular in $\tau$ and quasi-periodic in $\mathfrak{z}$ and is invariant under the action of the Weyl group $W(R)$ on the lattice variable $\mathfrak{z}$. All $W(R)$-invariant weak Jacobi forms of integral weight and integral index form a bigraded algebra $J_{*,R,*}^{\w, W(R)}$ over $\CC$. When $R$ is not of type $E_8$, Wirthm\"{u}ller showed that $J_{*,R,*}^{\w, W(R)}$ is a polynomial algebra generated by $r+1$ Jacobi forms over the ring $M_*(\SL_2(\ZZ))=\CC[E_4, E_6]$ of $\SL_2(\ZZ)$-modular forms. For example, the algebra of $W(A_1)$-invariant weak Jacobi forms, that is, the ring of even-weight weak Jacobi forms introduced by Eichler and Zagier, is freely generated by two forms of index $1$ and weight $0$ and $-2$. 
Recently, the second named author proved in \cite{Wan21a} that the ring of $W(E_8)$-invariant weak Jacobi forms is not a polynomial algebra. In other words, there exist some algebraic relations among generators. It is still unknown if this ring is finitely generated. In the study on the Seiberg--Witten curve of $E$-string theory \cite{Eguchi:2002gjr}, Sakai \cite{Sak17} constructed nine $W(E_8)$-invariant holomorphic Jacobi forms denoted $A_1$, $A_2$, $A_3$, $A_4$, $A_5$, $B_2$, $B_3$, $B_4$, $B_6$. The forms $A_i$ have weight $4$ and index $i$, and they reduce to the Eisenstein series $E_4$ when $\mathfrak{z}=0$. The forms $B_j$ have weight $6$ and index $j$, and they reduce to $E_6$ when $\mathfrak{z}=0$. In \cite{Sak19} Sakai conjectured that for any $W(E_8)$-invariant Jacobi form $\phi$ there exists a $\SL_2(\ZZ)$-modular form $f(\tau)$ such that the product $f\phi$ can be written as a polynomial in these $A_i$, $B_j$ and $E_4$, $E_6$. This conjecture was proved by the second named author in \cite{Wan21a}. In this paper we will determine the best possible $f$ for arbitrary index and give some further applications. Our description is inspired by some conjectures in string theory. Let us briefly introduce the physical background.

The $E$-string theory is a typical 6d $(1,0)$ superconformal field theory (SCFT) with $E_8$ flavor symmetry \cite{Witten:1995gx,Seiberg:1996vs,Ganor:1996mu}. In the $S^1/\mathbb{Z}_2$ compactification of M-theory, $m$ $E$-strings are realized by $m$ M2-branes stretched between a M5-brane and a M9-brane. The bound state of $m$ $E$-strings are captured by topologically twisted 4d $\mathcal{N}=4$ $U(m)$ Yang-Mills theories on half K3, which is an elliptic surface realized as $\mathbb{P}^2$ with nine points blown up \cite{Klemm:1996hh,Minahan:1998vr}. From the view point of Calabi--Yau geometry, the $E$-string theory is equivalent to topological string theory on local half K3 Calabi--Yau threefold. By $S^1$ compactification, $E$-string theory gives marginal 5d $\mathcal{N}=1$ $SU(2)$ gauge theory with eight fundamentals, which can flow to almost all 5d rank-one SCFTs \cite{Seiberg:1996bd}. $E$-string theory is also closely connected to van Diejen integrable model and elliptic Painlev\'e system \cite{Nazzal:2018brc,noumi2020elliptic}. All these relations make $E$-string theory the mother of almost all ``genus one'' theories in the sense of mirror/spectral curves, or ``rank one'' theories in the sense of Coulomb branch dimension. 

A general spacetime setting which preserves the supersymmetry for $E$-string theory is the so called 6d Omega background $(\mathbb{C}^2\times T^2)_{\epsilon_1,\epsilon_2}$, where each $\mathbb{C}$ plane is rotated by $z_i\to e^{\epsilon_i}z_i$ around the cycles of the torus. For such system, the $\epsilon_{1,2}$ expansion of the total instanton free energy $F=\sum_{n,g,m\ge0}Q^m(-\epsilon_1\epsilon_2)^{g-1}(\epsilon_1+\epsilon_2)^{2n}F_{(n,g,m)}$ defines the refined free energies of $E$-strings, where $Q$ counts the number of strings. The situation when keeping all $E_8$ flavor fugacities $\mathfrak{z}$ is often called massive, while turning $\mathfrak{z}=0$ massless. It is generally conjectured that the refined free energies $F_{(n,g,m)}$ of $m$ $E$-strings in the massive case are $W(E_8)$-invariant quasi Jacobi forms of index $m$. The scaled refined free energies $( \frac{\eta^{12}}{\sqrt{q}} )^m F_{(n,g,m)}$ of $E$-strings in the massless case were found to be quasi modular forms of weight $2(n+g)+6m-2$ on $\SL_2(\ZZ)$ and thus can be expressed as polynomials in $E_2$, $E_4$ and $E_6$. It is then natural to guess that the scaled refined free energies $( \frac{\eta^{12}}{\sqrt{q}} )^m F_{(n,g,m)}$ in the massive case can be written as polynomials in $E_2$, $E_4$, $E_6$, $A_i$ and $B_j$. The process to fix all the polynomial coefficients to determine $F_{(n,g,m)}$ is called modular bootstrap and the ansatz is called modular ansatz. In \cite{HKP13}, Huang, Klemm and Poretschkin developed the refined modular anomaly equation to efficiently compute $F_{(n,g,m)}$ of $E$-strings. They found that the naive modular ansatz is true for index $m\leq 4$, but usually not true for $m\geq 5$. Indeed, they found some $(\frac{\eta^{12}}{\sqrt{q}})^5 F_{(n,g,5)}$ which cannot be expressed as polynomials in Sakai's nine forms unless further multiplied by $E_4$. Later Del Zotto, Gu, Huang, Kashani-Poor, Klemm and Lockhart \cite{ZGH+18} discovered an exceptional $W(E_8)$-invariant holomorphic Jacobi form of weight $16$ and index $5$ defined by the polynomial
\begin{equation}\label{eq:P5}
    P_{16,5}= 864 A_1^3 A_2 + 3825 A_1 B_2^2 - 770 A_3 B_2 E_6 - 840 A_2 B_3 E_6 + 60 A_1 B_4 E_6 + 21 A_5 E_6^2.
\end{equation}
They checked numerically that $P_{16,5}$ vanishes at the zero points of $E_4$ for general lattice variable $\mathfrak{z}$ and then conjectured that the quotient $P_{16,5}/E_4$ is holomorphic. They did not find other similar polynomials,  so they further conjectured that any Jacobi form expressed as a polynomial in $A_i$, $B_j$ and $E_6$ which vanishes at the zero points of $E_4$ must be divisible by the above polynomial. In this paper we will prove their conjectures. 

In a similar manner, Sakai's nine $A_i,B_j$ forms are also used in the modular bootstrap of the elliptic genera of $E$-strings and $E_8\times E_8$ Heterotic strings \cite{Haghighat:2014pva,Cai:2014vka,Gu:2017ccq}. One of the main goals of the current work is to establish a rigorous foundation for modular bootstrap whenever $E_8$ symmetry, no matter flavor or gauge, is involved.

In \cite{Wan21b} the second named author established the modular Jacobian criterion to give an automorphic proof of Wirthm\"{u}ller's theorem. We will use this approach and the distinguished Jacobi form $P_{16,5}/E_4$ to prove the following theorem, which gives a full description of $W(E_8)$-invariant Jacobi forms in terms of Sakai's forms. 

\begin{theorem}\label{th:Main}
\noindent
\begin{enumerate}
    \item The quotient ${P_{16,5}}/{E_4}$ is a $W(E_8)$-invariant holomorphic Jacobi form of weight $12$ and index $5$. 
    \item For any $W(E_8)$-invariant Jacobi form $P\in \CC[E_6, A_1,A_2,B_2,A_3,B_3,A_4,B_4,A_5,B_6]$, if ${P}/{E_4}$ is holomorphic on $\HH\times (E_8\otimes\CC)$, then 
    $$
    \frac{P}{P_{16,5}} \in \CC[E_6, A_1,A_2,B_2,A_3,B_3,A_4,B_4,A_5,B_6].
    $$
    \item Every $W(E_8)$-invariant weak Jacobi form of index $t$ can be expressed uniquely as 
    \begin{equation}\label{eq:universal}
        \frac{\sum_{j=0}^{t_1} P_j E_4^j P_{16, 5}^{t_1-j} }{\Delta^{N_t}E_4^{t_1}},
    \end{equation}
    where 
    \begin{itemize}
        \item[(i)] $t_1$ is the integer part of ${t}/{5}$;
        \item[(ii)] $P_{t_1}\in \CC[E_4,E_6,A_1,A_2,B_2,A_3,B_3,A_4,B_4,A_5,B_6]$;
        \item[(iii)] $P_j\in \CC[E_6,A_1,A_2,B_2,A_3,B_3,A_4,B_4,A_5,B_6]$ for $0\leq j<t_1$;
        \item[(iv)] $N_t$ is defined as follows
        \begin{equation*}
        N_t = \begin{cases}
        5t_0,   &\text{if $t=6t_0$ or $6t_0+1$},\\
        5t_0+1,   &\text{if $t=6t_0+2$},\\
        5t_0+2,   &\text{if $t=6t_0+3$},\\
        5t_0+3,   &\text{if $t=6t_0+4$ or $6t_0+5$}.\\
      \end{cases}    
\end{equation*}
    \end{itemize}
\end{enumerate}
\end{theorem}
The powers $t_1$ and $N_t$ here are sharp. There exist $W(E_8)$-invariant Jacobi forms of arbitrary index which cannot be expressed in the form \eqref{eq:universal} if we replace $t_1$ or $N_t$ with any smaller integer. Besides,
our theorem implies that there is no $W(E_8)$-invariant Jacobi form which vanishes at the zero points of $E_6$ and lies in the ring
$\CC[E_4,A_1,A_2,B_2,A_3,B_3,A_4,B_4,A_5,B_6]$. This confirms the numerical search in \cite{ZGH+18}.

The function \eqref{eq:universal} is a weak Jacobi form if and only if the numerator has Fourier expansion of the form 
\begin{equation}\label{eq:system}
    \sum_{j=0}^{t_1} P_j E_4^j P_{16, 5}^{t_1-j} = O(q^{N_t}), \quad q=e^{2\pi i \tau}.
\end{equation}
We can find a basis of the space of $W(E_8)$-invariant weak Jacobi forms of given weight and given index by solving the system of linear equations defined by \eqref{eq:system}. In \cite{Wan21a} the second named author proved that the space $J_{*,E_8,t}^{\w, W(E_8)}$ of $W(E_8)$-invariant weak Jacobi forms of integral weight and given index $t$ is a free module over $M_*(\SL_2(\ZZ))$ whose rank is determined by a generating series. He also determined the generators of $J_{*,E_8,t}^{\w, W(E_8)}$ for $t\leq 4$ using the differential operators on Jacobi forms. In this paper we calculate the Fourier expansions of Sakai's forms up to $q^9$-terms. By solving linear systems of type \eqref{eq:system}, we figure out the structure of $J_{*,E_8,t}^{\w, W(E_8)}$ for $t\leq 13$. We also successfully determine the space of $W(E_8)$-invariant holomorphic Jacobi forms of arbitrary weight and index $t\leq 11$. The module $J_{*,E_8,t}^{\w, W(E_8)}$ is quite complicated when the index $t$ is large, such as the largest module $J_{*,E_8,13}^{\w, W(E_8)}$ we have determined has $364$ generators. 

The paper is organized as follows. In \S \ref{sec:def} we give a brief introduce of  $W(E_8)$-invariant Jacobi forms. \S \ref{sec:proof} is devoted to the proof of Theorem \ref{th:Main}. We calculate the generating series of $J_{*,E_8,t}^{\w, W(E_8)}$ for index $t\leq 13$ in \S \ref{sec:weak} and present their Laurent expansions in Appendix I.  We discuss $W(E_8)$-invariant holomorphic Jacobi forms in \S \ref{sec:holomorphic} and construct three exceptional generators in Appendix II. In \S \ref{sec: conj} we propose five conjectures on Jacobi forms and provide some evidence.

\section{Preliminaries}\label{sec:def}
In this section we define $W(E_8)$-invariant Jacobi forms. Let $\NN$ be the set of non-negative integers and $(-,-)$ be the standard scalar product on $\RR^8$.  Let $W(E_8)$ denote the Weyl group of $E_8$. We use the model, simple roots and fundamental weights of $E_8$ fixed in \cite[\S 3.1]{Wan21a}. We say that a vector $v\in E_8$ has norm $n$ if $(v,v)=2n$. 
\begin{definition}
Let $\varphi : \HH \times (E_8 \otimes \CC) \rightarrow \CC$ be a holomorphic function and $k\in \ZZ$, $t\in \NN$. If $\varphi$ satisfies the following properties
\begin{itemize}
\item[(i)] Weyl invariance:
\begin{equation*}
\varphi(\tau, \sigma(\mathfrak{z}))=\varphi(\tau, \mathfrak{z}), \quad \sigma\in W(E_8),
\end{equation*}
\item[(ii)] Quasi-periodicity:
\begin{equation*}
\varphi (\tau, \mathfrak{z}+ x \tau + y)= \exp\left(-t\pi i [ (x,x)\tau +2(x,\mathfrak{z}) ]\right) \varphi ( \tau, \mathfrak{z} ), \quad x,y\in E_8,
\end{equation*}
\item[(iii)] Modularity:
\begin{equation*}
\varphi \left( \frac{a\tau +b}{c\tau + d},\frac{\mathfrak{z}}{c\tau + d} \right) = (c\tau + d)^k \exp\left( t\pi i \frac{c(\mathfrak{z},\mathfrak{z})}{c \tau + d}\right) \varphi ( \tau, \mathfrak{z} ), \quad \left( \begin{array}{cc}
a & b \\ 
c & d
\end{array} \right)   \in \SL_2(\ZZ),
\end{equation*}
\item[(iv)] $\varphi ( \tau, \mathfrak{z} )$ has the Fourier expansion 
\begin{equation*}
\varphi ( \tau, \mathfrak{z} )= \sum_{ n=0 }^\infty \sum_{ \ell \in E_8}f(n,\ell)e^{2\pi i (n\tau + (\ell,\mathfrak{z}))},
\end{equation*}
\end{itemize}
then $\varphi$ is called a $W(E_8)$-invariant weak Jacobi form of weight $k$ and index $t$. If $f(n,\ell)= 0$ whenever $2nt - (\ell,\ell) <0$, then $\varphi$ is called a $W(E_8)$-invariant holomorphic Jacobi form.  
\end{definition}

Every non-zero $W(E_8)$-invariant weak Jacobi form has even weight. The $W(E_8)$-invariant weak Jacobi forms of index $0$ do not depend on the lattice variable $\mathfrak{z}$ and are actually modular forms on $\SL_2(\ZZ)$. The quasi-periodicity implies that in the above Fourier expansion $f(n_1,\ell_1)=f(n_2,\ell_2)$ if $2n_1 t-(\ell_1,\ell_1)=2n_2 t-(\ell_2,\ell_2)$ and if $\ell_1-\ell_2\in t E_8$. 
We denote the vector spaces of $W(E_8)$-invariant weak and holomorphic Jacobi forms of weight $k$ and index $t$ respectively by
$$
J^{\w ,W(E_8)}_{k,E_8,t} \supsetneq  J^{W(E_8)}_{k,E_8,t}.
$$
We will investigate the free $M_*(\SL_2(\ZZ))$-modules
$$
J_{*,E_8,t}^{\w, W(E_8)}:= \bigoplus_{k\in \ZZ} J^{\w ,W(E_8)}_{k,E_8,t}, \quad J_{*,E_8,t}^{W(E_8)}:= \bigoplus_{k\in \ZZ} J^{W(E_8)}_{k,E_8,t}
$$
and the bigraded algebra 
$$
J_{*,E_8,*}^{\w, W(E_8)}:= \bigoplus_{t=0}^\infty J_{*,E_8,t}^{\w, W(E_8)}.
$$

Sakai's forms $A_1$, $A_2$, $B_2$, $A_3$, $B_3$, $A_4$, $B_4$, $A_5$, $B_6$ were first constructed in \cite[Appendix A.1]{Sak17}. 
It was proved in \cite[Theorem 4.1]{Wan21a} that the nine Jacobi forms $A_i$ and $B_j$ are algebraically independent over $\CC[E_4, E_6]$. Due to its importance, we briefly explain how Sakai constructed these forms. One starts with the Jacobi theta function of $E_8$
$$
\vartheta_{E_8}(\tau,\mathfrak{z})=\sum_{\ell \in E_8} e^{\pi i (\ell,\ell)\tau + 2\pi i (\ell, \mathfrak{z})}
$$
which is the unique $W(E_8)$-invariant holomorphic Jacobi form of weight $4$ and index $1$. Acting the index raising Hecke operators $T_{-}(t)$ on $\vartheta_{E_8}$ (see e.g. \cite[Lemma 3.6]{Wan21a}), one obtains $W(E_8)$-invariant holomorphic Jacobi forms of weight $4$ and arbitrary index $t$
\begin{equation}\label{eq:X_t}
    X_t(\tau, \mathfrak{z}) = 1 + O(q), \quad q=e^{2\pi i\tau}. 
\end{equation}
Sakai's forms $A_j$ are constructed as
\begin{equation}
    A_j(\tau, \mathfrak{z})=X_j(\tau, \mathfrak{z}), \; j=1,2,3,5, \quad A_4(\tau, \mathfrak{z})= \vartheta_{E_8}(\tau, 2\mathfrak{z}).
\end{equation}
To construct $B_t$, one first takes an appropriate modular form $g_t$ of weight $2$ on the congruence subgroup $\Gamma_0(t)$ of $\SL_2(\ZZ)$. Then the trace sum of $g_t(\tau)\vartheta_{E_8}(t\tau, t\mathfrak{z})$ with respect to the cosets of $\Gamma_0(t)\setminus \SL_2(\ZZ)$ defines a $W(E_8)$-invariant holomorphic Jacobi form of weight $6$ and index $t$. That is the desired $B_t$.

The Fourier expansion characterizes Jacobi forms. Due to the Weyl invariance, the Fourier expansion of any $W(E_8)$-invariant Jacobi form can be expressed in terms of Weyl orbits of vectors in $E_8$. We review some useful facts following \cite{Wan21a}. For any $v\in E_8$ we define the Weyl orbit of $v$ as
\begin{equation}
\orb(v)=\sum_{\sigma\in W(E_8)/W(E_8)_{v}}e^{2\pi i (\sigma(v),\mathfrak{z})},
\end{equation}
where $W(E_8)_{v}$ is the stabilizer subgroup of $W(E_8)$ with respect to $v$. For any non-negative integer $n$ the $q^n$-term of $\varphi$ 
$$
[\varphi]_{q^n}=\sum_{\ell\in E_8}f(n,\ell)e^{2\pi i (\ell,\mathfrak{z})}
$$
can be written as a $\CC$-linear combination of Weyl orbits. Let $\alpha_i$ and $w_i$, $1\leq i \leq 8$, be the simple roots and fundamental weights of $E_8$ respectively (see \cite[\S 3.1]{Wan21a} for their coordinates). The eight fundamental Weyl orbits $\orb(w_i)$ are algebraically independent over $\CC$. Moreover, every Weyl orbit $\orb(v)$ can be expressed as a polynomial in these fundamental $\orb(w_i)$. By \cite{Bou68}, every Weyl orbit meets the set $\Lambda_{+}$ in exactly one point, where
$$
\Lambda_{+}=\left\{m\in E_8: (\alpha_i,m)\geq 0, 1\leq i \leq 8 \right\}=\left\{m=\sum_{i=1}^8 m_i w_i: m_i \in \NN, 1\leq i\leq 8 \right \}
$$
is the closure of a Weyl chamber. Thus we only need to consider the Weyl orbits of vectors in $\Lambda_{+}$. 
With respect to the partial order on $E_8$, we have the decomposition
$$
\orb(m)=\prod_{i=1}^8 \orb(w_i)^{m_i} + \sum_{\substack{l\in \Lambda_{+} \\ l<m}} c_{l,m} \orb(l),
$$
where $c_{l,m}$ are some integers. Let us define 
\begin{equation}\label{eq:T}
    T(m):=(m,w_8)=2m_1+3m_2+4m_3+6m_4+5m_5+4m_6+3m_7+2m_8.
\end{equation}
Since $w_8$ is the highest root of $E_8$, the condition $l\leq m$ implies that $T(l)\leq T(m)$. We then derive the further decomposition
\begin{equation}\label{eq:decomposition}
    \orb(m)=\sum_{\substack{l\in \Lambda_{+} \\ T(l)\leq T(m)}} c_l\prod_{i=1}^8 \orb(w_i)^{l_i}, \quad c_m=1.
\end{equation}

In this paper we first compute the Fourier expansions of Sakai's forms in terms of Weyl orbits. This can be calculated efficiently by their definitions. We then use \eqref{eq:decomposition} to express their Fourier coefficients as polynomials in the eight fundamental Weyl orbits. This expression is very convenient for calculating the Fourier expansions of monomials in Sakai's generators. The $q^0$-term plays a crucial role in the study of weak Jacobi forms. We recall a useful result. \cite[Lemma 4.2]{Wan21a} states that the $q^0$-term of any $W(E_8)$-invariant weak Jacobi form of index $t$ can be expressed as 
\begin{equation}\label{eq:q^0}
\sum_{\substack{m\in \Lambda_{+} \\ T(m)\leq t}} c_m\prod_{i=1}^8 \orb(w_i)^{m_i}.
\end{equation}
The number of monomials including the constant term in the above sum is equal to the rank of the free module $J_{*,E_8,t}^{\w,W(E_8)}$ (see \cite[Theorem 4.1]{Wan21a} and its proof). This number also equals the number of orbits of the quotient group $E_8/tE_8$ under the action of $W(E_8)$, and $\{m\in \Lambda_{+} : T(m)\leq t \}$ gives a representative set of these orbits. 

Finally, we describe the number of elements in a Weyl orbit $\orb(m)$, which equals the index of the stabilizer $W(E_8)_m$ in $W(E_8)$. We remind that Weyl orbits of the same norm do not necessarily have distinct number of elements (see Appendix II for some examples). 

\begin{lemma}
Let $m=\sum_{i=1}^8 m_i w_i\in \Lambda_{+}$. Then the stabilizer of $m$ in $W(E_8)$ is given by the intersection of all stabilizers of fundamental weights $w_i$ with $m_i\neq 0$. 
\end{lemma}
\begin{proof}
Let $g\in W(E_8)$. Then $g(m)=m$ if and only if 
$
\sum_{i=1}^8 m_i g(w_i) = \sum_{i=1}^8 m_i w_i. 
$
With respect to the partial order on $E_8$, we have that $g(w_i)\leq w_i$ for $1\leq i\leq 8$ (see \cite{Bou68}). Therefore, $g(m)=m$ if and only if $g(w_i)=w_i$ for all $i$ such that $m_i\neq 0$.  
\end{proof}

\section{The proof of Theorem \ref{th:Main}}\label{sec:proof}
In this section we prove Theorem \ref{th:Main}. We divide its proof into several lemmas. 

As explained in \cite[\S 3.2 p.529--530]{Wan21a},  we take  $g(\tau)=(5E_2(5\tau)-E_2(\tau))/4$ which is a modular form of weight $2$ on the congruence subgroup $\Gamma_0(5)$, and define $\widehat{B}_5$ analogous to Sakai's $B_j$ as
\begin{equation}
    \widehat{B}_5(\tau,\mathfrak{z})=\frac{5^4}{5^4-1}\bigg[g(\tau)\vartheta_{E_8}(5\tau,5\mathfrak{z})-\frac{1}{5^5}\sum_{k=0}^4 g\Big(\frac{\tau+k}{5}\Big)\vartheta_{E_8}\Big(\frac{\tau+k}{5},\mathfrak{z}\Big)\bigg].
\end{equation}
We note that $\widehat{B}_5$ is a $W(E_8)$-invariant holomorphic Jacobi form of weight $6$ and index $5$ whose reduction is $\widehat{B}_5(\tau,0)=E_6$.

\begin{lemma}\label{lem:B5}
The Jacobi form $\widehat{B}_5$ satisfies the identity
\begin{equation}
    179712\Delta E_4 \widehat{B}_5=E_6P_{16,5}+E_4Q_{18,5},
\end{equation}
where 
\begin{align*}
Q_{18,5} =& -2880 A_1^3 B_2 + 1350 A_1 A_2 B_2 E_4 + 1920 A_1^2 B_3 E_4 - 70 A_3 B_2 E_4^2 - 
 600 A_2 B_3 E_4^2 - 60 A_1 B_4 E_4^2 \\
 & - 567 A_1 A_2^2 E_6 + 672 A_1^2 A_3 E_6 - 2400 B_2 B_3 E_6 - 504 A_2 A_3 E_4 E_6 - 21 A_5 E_4^2 E_6.
\end{align*}
\end{lemma}
\begin{proof}
Both sides of the above identity are $W(E_8)$-invariant Jacobi forms of weight 22 and index 5. We check by computer that their Fourier coefficients are the same up to $q^4$-terms. Thus their difference divided by $\Delta^4$ defines a $W(E_8)$-invariant weak Jacobi form of weight $-26$ and index 5, which has to be zero because the weight of a non-zero $W(E_8)$-invariant weak Jacobi form of index 5 is at least $-18$ (see \cite[Proposition 5.17]{Wan21a}). This proves the desired identity.
\end{proof}

We remark that $P_{16,5}/E_4^2$ is non-holomorphic because $P_{16,5}$ reduces to $864E_4^4+2296E_4E_6^2$ when $\mathfrak{z}=0$. 

In \cite{Wan21b} the second named author defined the Jacobian of Jacobi forms and used this tool to give a simple proof of  Wirthm\"{u}ller's theorem. We refer to \cite[Proposition 2.2, Proposition 2.3]{Wan21b} for details. We will apply this method to the case of $E_8$. Let us first calculate the Jacobian of Sakai's generators defined as the determinant
\begin{equation}
\left\lvert \begin{array}{ccccccccc}
A_1 & 2A_2 & 2B_2 & 3A_3 & 3B_3 & 4A_4 & 4B_4& 5A_5 & 6B_6 \\ 
\frac{1}{2\pi i}\frac{\partial A_1}{\partial z_1} & \frac{1}{2\pi i}\frac{\partial A_2}{\partial z_1}& \frac{1}{2\pi i}\frac{\partial B_2}{\partial z_1} &
\frac{1}{2\pi i}\frac{\partial A_3}{\partial z_1} &
\frac{1}{2\pi i}\frac{\partial B_3}{\partial z_1} &
\frac{1}{2\pi i}\frac{\partial A_4}{\partial z_1} &
\frac{1}{2\pi i}\frac{\partial B_4}{\partial z_1} &
\frac{1}{2\pi i}\frac{\partial A_5}{\partial z_1} &
\frac{1}{2\pi i}\frac{\partial B_6}{\partial z_1} 
\\ 
%\frac{1}{2\pi i}\frac{\partial A_1}{\partial z_2} & \frac{1}{2\pi i}\frac{\partial A_2}{\partial z_2}& \frac{1}{2\pi i}\frac{\partial B_2}{\partial z_2} &
%\frac{1}{2\pi i}\frac{\partial A_3}{\partial z_2} &
%\frac{1}{2\pi i}\frac{\partial B_3}{\partial z_2} &
%\frac{1}{2\pi i}\frac{\partial A_4}{\partial z_2} &
%\frac{1}{2\pi i}\frac{\partial B_4}{\partial z_2} &
%\frac{1}{2\pi i}\frac{\partial A_5}{\partial z_2} &
%\frac{1}{2\pi i}\frac{\partial B_6}{\partial z_2} 
%\\
\vdots & \vdots & \vdots & \vdots & \vdots & \vdots & \vdots & \vdots & \vdots \\ 
\frac{1}{2\pi i}\frac{\partial A_1}{\partial z_8} & \frac{1}{2\pi i}\frac{\partial A_2}{\partial z_8}& \frac{1}{2\pi i}\frac{\partial B_2}{\partial z_8} &
\frac{1}{2\pi i}\frac{\partial A_3}{\partial z_8} &
\frac{1}{2\pi i}\frac{\partial B_3}{\partial z_8} &
\frac{1}{2\pi i}\frac{\partial A_4}{\partial z_8} &
\frac{1}{2\pi i}\frac{\partial B_4}{\partial z_8} &
\frac{1}{2\pi i}\frac{\partial A_5}{\partial z_8} &
\frac{1}{2\pi i}\frac{\partial B_6}{\partial z_8} 
\end{array}   \right\rvert,
\end{equation}
where $z_i,i=1,\dots,8$, are the standard basis of $\mathbb{R}^8$.
\begin{lemma}\label{lem:172}
The modular Jacobian of Sakai's forms satisfies the identity
\begin{equation}
    J:=J(A_1,A_2,B_2,A_3,B_3,A_4,B_4,A_5,B_6)= c \Delta^{14}E_4 \cdot \Phi_{E_8},\qquad c=-\frac{3^3}{5^47^2},
\end{equation}
where $\Phi_{E_8}$ is the theta block associated to $E_8$ (see \cite{GSZ19} for the general theory of theta blocks developed by Gritsenko, Skoruppa and Zagier)
$$
\Phi_{E_8}(\tau,\mathfrak{z})=\prod_{r}\frac{\vartheta(\tau, (r,\mathfrak{z}))}{\eta^3(\tau)}.
$$
The above product takes over all positive roots of $E_8$ and $\vartheta$ is the odd Jacobi theta function
$$
\vartheta(\tau,z)=q^{\frac{1}{8}}(\zeta^{\frac{1}{2}}-\zeta^{-\frac{1}{2}})\prod_{n=1}^\infty (1-q^n\zeta)(1-q^n\zeta^{-1})(1-q^n), \quad q=e^{2\pi i\tau}, \; \zeta=e^{2\pi i z}. 
$$
\end{lemma}

\begin{proof}
By \cite[\S 4]{Wan21b}, $\Phi_{E_8}$ is a weak $E_8$ Jacobi form of weight $-120$ and index $30$ which is anti-invariant under the action of $W(E_8)$ (i.e. invariant up to the determinant character). Moreover, it vanishes precisely on 
$$
\{ (\tau, \mathfrak{z})\in \HH\times (E_8\otimes \CC): (r,\mathfrak{z}) \in \ZZ + \ZZ\tau \; \textit{for some positive root $r$}\}
$$
with multiplicity one. By \cite[Proposition 2.2]{Wan21b}, $J$ also vanishes on the above set. Thus $J/\Phi_{E_8}$ is holomorphic and defines a $W(E_8)$-invariant weak Jacobi form of weight 172 and index 0, which yields that $J/\Phi_{E_8}$ is a $\SL_2(\ZZ)$-modular form of weight $172$. We notice that every partial derivative of $A_i$ or $B_j$ cancels the $q^0$-term. By calculating the Fourier expansions of $A_i$ and $B_j$ up to $q^7$-terms, it is sufficient to calculate the $q^{14}$-term of $J$. We find that $J/\Phi_{E_8}=c q^{14} + O(q^{15})$, which implies that $J/\Phi_{E_8}=c\Delta^{14}E_4$. This completes the proof. 
\end{proof}

\begin{remark}
By \cite[Proposition 2.2 (2)]{Wan21b}, the above Jacobian $J$ is not identically zero if and only if the forms $A_i$ and $B_j$ are algebraically independent over $\CC[E_4,E_6]$. Thus Lemma \ref{lem:172} yields a new proof of the algebraic independence of Sakai's forms.
\end{remark}

\begin{lemma}\label{lem:pure-delta}
For any $W(E_8)$-invariant weak Jacobi form $\phi$, there exists an integer $N$ such that 
$$
\Delta^N \phi \in \CC[E_4,E_6,A_1,A_2,B_2,A_3,B_3,A_4,B_4,\widehat{A}_5,B_6],
$$
where 
$$
\widehat{A}_5:= \frac{P_{16,5}}{E_4}-21E_4^2A_5. 
$$
\end{lemma}
\begin{proof}
From the definition of the modular Jacobian, we see that
\begin{align*}
    J(A_1,A_2,B_2,A_3,B_3,A_4,B_4,\widehat{A}_5,B_6)=&J(A_1,A_2,B_2,A_3,B_3,A_4,B_4,P_{16,5}-21E_4^3A_5,B_6)/E_4\\
    =&-21\cdot 1728 \Delta J(A_1,A_2,B_2,A_3,B_3,A_4,B_4,A_5,B_6) / E_4. 
\end{align*}
By Lemma \ref{lem:172}, we have
$$
J(A_1,A_2,B_2,A_3,B_3,A_4,B_4,\widehat{A}_5,B_6)/ \Phi_{E_8} = -36288 c \Delta^{15}.
$$
We then prove the claim by applying the criterion established in \cite[Proposition 2.3 (2)]{Wan21b}.
\end{proof}

We also need the following eight specific $W(E_8)$-invariant weak Jacobi forms of weight $4$ whose $q^0$-terms are a single fundamental Weyl orbit. 

\begin{lemma}\label{lem:q0}
The following $W(E_8)$-invariant weak Jacobi forms $\varphi_{-,t}$ of weight $4$ and index $t$ with indicated $q^0$-term exist.
\begin{align*}
    \varphi_{4a,2}&= P_1(E_4,E_6,A_i,B_j)/ \Delta =  \orb(w_1) +O(q),\\
    \varphi_{4b,2}&= P_2(E_4,E_6,A_i,B_j)/ \Delta =  \orb(w_8) +O(q),\\
        \varphi_{4a,3}&= P_3(E_4,E_6,A_i,B_j)/ \Delta =  \orb(w_7) +O(q),\\
    \varphi_{4b,3}&= P_4(E_4,E_6,A_i,B_j)/ \Delta^2 =  \orb(w_2) +O(q),\\
    \varphi_{4a,4}&= P_5(E_4,E_6,A_i,B_j)/ \Delta^3 =  \orb(w_3) +O(q),\\
    \varphi_{4b,4}&= P_6(E_4,E_6,A_i,B_j)/ \Delta^3 =  \orb(w_6) +O(q),\\
    \varphi_{4,5}&= P_7(E_4,E_6,A_i,B_j)/ \Delta^3 =  \orb(w_5) +O(q),\\
    \varphi_{4,6}&= P_8(E_4,E_6,A_i,B_j)/ \Delta^5 =  \orb(w_4) +O(q),
\end{align*}
where 
\begin{align*}
   P_1 =&\, \frac{1}{4} \left(-12 A_1^2 E_4^2+17 A_2 E_4^3+10 A_2 E_6^2-15 B_2 E_4 E_6\right),\\
   P_2 =&\, \frac{1}{72} (24 A_1^2 E_4^2-14 A_2 E_4^3+5 A_2 E_6^2-15 B_2 E_4 E_6),\\
P_3 =&\, \frac{7}{18} (-27A_1A_2E_4^2-45A_1B_2E_6+37A_3E_4^3+35A_3E_6^2),\\
   P_4 =&\,  \frac{1}{864} (126 A_1^3 E_4^4-414 A_1^3 E_4 E_6^2+675 A_1 A_2 E_4^2 E_6^2-243 A_1 A_2 E_4^5-1440 A_1 B_2 E_6^3\\
   &\phantom{--}\,-251 A_3 E_4^3 E_6^2+122 A_3 E_4^6+465 A_3 E_6^4-20 B_3 E_4^4 E_6+980 B_3 E_4 E_6^3).
\end{align*}
The other polynomials are very long and we omit their expressions here. When the index is greater than $2$ the above polynomial is not unique because the associated space $J_{-8,E_8,t}^{\w, W(E_8)}$ is non-trivial. 
\end{lemma}

\begin{proof}[Proof of Theorem \ref{th:Main}] 
(1) This is a direct consequence of Lemma \ref{lem:B5}. 

(2) It follows from Lemma \ref{lem:pure-delta}. 

(3) Let $\varphi_t$ be a $W(E_8)$-invariant weak Jacobi form of index $t$. By Lemma \ref{lem:pure-delta}, there exists an integer $N$ such that 
\begin{equation}\label{eq:varphi}
    \varphi_t=\frac{\sum_{j=0}^{t_1} P_j E_4^j P_{16, 5}^{t_1-j} }{\Delta^{N}E_4^{t_1}}.
\end{equation}
To prove the third assertion, it suffices to show that one can choose $N$ satisfying $N\leq N_t$. Recall that the space of all $W(E_8)$-invariant weak Jacobi forms of fixed index $t$ is a free module of rank $r(t)$ over $M_*(\SL_2(\ZZ))$ (see \cite[Theorem 4.1]{Wan21a}). It is enough to prove that one can always choose $N\leq N_t$ such that \eqref{eq:varphi} holds for every generator of the free module. Let $\phi_j$, $1\leq j \leq r(t)$, be the generators of the free module $J_{*,E_8,t}^{\w,W(E_8)}$. Clearly, every $\phi_j$ can be expressed in the form \eqref{eq:varphi}. We denote by $M_j$ the smallest integer such that \eqref{eq:varphi} holds for $\phi_j$. Assume that $M_1$ is the largest of all $M_j$. Obviously, every generator $\phi_j$ has non-zero $q^0$-term. Moreover, the $q^0$-term of every $\phi_j$ is a polynomial in the eight fundamental Weyl orbits satisfying the restriction defined in \eqref{eq:q^0}. By \cite[Lemma 4.2]{Wan21a} and the construction of forms in Lemma \ref{lem:q0}, for any sufficiently large integer $D$ there exists a $W(E_8)$-invariant weak Jacobi form $\psi_t$ of index $t$ in the ring
$$
\CC[E_4, E_6, A_1, \varphi_{4a,2}, \varphi_{4b,2}, \varphi_{4a,3}, \varphi_{4b,3}, \varphi_{4a,4}, \varphi_{4b,4}, \varphi_{4,5}, \varphi_{4,6}]
$$
such that the $q^0$-term of the difference $E_4^D\phi_1 - \psi_t$ is zero. Therefore, $(E_4^D\phi_1 - \psi_t)/\Delta$ is a $W(E_8)$-invariant weak Jacobi form of index $t$ and then a $\CC[E_4,E_6]$-linear combination of the generators $\phi_j$. It follows that the number $M_1$ defined above is exactly the smallest integer $N$ appearing in the expression of type \eqref{eq:varphi} for $\psi_t$. We then derive the formula of $N_t$ from the powers of $\Delta$ in the construction of basic forms in Lemma \ref{lem:q0}. \end{proof}

\section{Free modules of weak Jacobi forms of given index}\label{sec:weak}
It was proved in \cite[Theorem 4.1]{Wan21a} that the space
$$
J_{*,E_8,t}^{\w ,W(E_8)}:=\bigoplus_{k\in \ZZ} J_{k,E_8,t}^{\w ,W(E_8)}
$$
of $W(E_8)$-invariant weak Jacobi forms of integral weight and given index $t$ is a free module over $M_*(\SL_2(\ZZ))$ and the rank $r(t)$ is given by the generating series
\begin{equation}\label{eq:rank}
\frac{1}{(1-x)(1-x^2)^2(1-x^3)^2(1-x^4)^2(1-x^5)(1-x^6)}=\sum_{t\geq 0}r(t)x^t.
\end{equation}
We formulate the first values of $r(t)$ in Table \ref{tab:rank}.
\begin{table}[ht]
\caption{The rank of $J_{*,E_8,t}^{\w ,W(E_8)}$ over $M_*(\SL_2(\ZZ))$}\label{tab:rank}
\renewcommand\arraystretch{1.5}
\noindent\[
\begin{array}{|c|c|c|c|c|c|c|c|c|c|c|c|c|c|c|c|c|c|c|}
\hline 
t & 1 & 2 & 3 & 4 & 5 & 6 & 7 & 8 & 9 & 10 & 11 & 12 & 13 & 14 &15 & 16 &17 &18 \\ 
\hline 
r(t) & 1 & 3 & 5 & 10 & 15 & 27 & 39 & 63 & 90 & 135 & 187 & 270 & 364 & 505 & 670 & 902 & 1173 & 1545  \\ 
\hline 
\end{array} 
\]
\end{table}

In \cite[\S 5]{Wan21a} the generators were determined and constructed when the index is less than $5$ by an approach based on the weight raising differential operators of Jacobi forms. In this section, as an application of Theorem \ref{th:Main}, we determine the weights of generators and construct the generators in terms of Sakai's forms $A_i$ and $B_j$ when the index is less than $14$. To this aim, we introduce the following algorithm. 

\begin{Algorithm}\label{algorithm}
\noindent
\begin{enumerate}
\item We determine a basis of $J_{k,E_8,t}^{\w,W(E_8)}$ for any even weight $k$.  The space $J_{k,E_8,t}^{\w,W(E_8)}$ is a finite-dimensional vector space over $\CC$. Every form in this space has an expression of form \eqref{eq:universal} which corresponds to a solution of the system of linear equations defined by the vanishing of $q^n$-terms ($0\leq n\leq N_t-1$) in the Fourier expansion 
\begin{equation}\label{eq:algorithm}
    \sum_{j=0}^{t_1} P_j E_4^j P_{16,5}^{t_1-j} = O(q^{N_t}) 
\end{equation}
in weight $k+12N_t+4t_1$. 
There are only finitely many linearly independent solutions in any fixed weight, because the weights and indices of these $P_j$ (polynomials in $E_4, E_6, A_i, B_j$) are bounded from above. 
\item We determine the minimal weight of non-zero $W(E_8)$-invariant weak Jacobi forms of given index $t$. If the equation \eqref{eq:algorithm} has only zero solution in weight $K+12N_t+4t_1$ and in weight $K+12N_t+4t_1-2$, then its solution is always trivial in any lower weight and thus $J_{k,E_8,t}^{\w,W(E_8)}=\{0\}$ when $k\leq K$.
\item We collect all $r(t)$ generators of the free module $J_{*,E_8,t}^{\w ,W(E_8)}$ from the minimal weight $K$ to the larger weight. 
\end{enumerate}
\end{Algorithm}

We first calculate the Fourier expansions of Sakai's forms $A_i$ and $B_j$ up to $q^9$-terms. The Fourier expansions involve $268$ Weyl orbits of vectors of norm $\frac{1}{2}(v,v)\leq 54$. We then express these Weyl orbits as polynomials in the eight fundamental Weyl orbits. Using the data and Algorithm \ref{algorithm}, we successfully determine all generators of $J_{*,E_8,t}^{\w ,W(E_8)}$ for $1\leq t\leq 13$.

\begin{theorem}
Let $d_{k,t}$ denote the number of generators of weight $k$ of $J_{*,E_8,t}^{\w,W(E_8)}$. For $1\leq t\leq 13$ the Laurent polynomials
$$
P^{\w}_t:=\sum_{k\in \ZZ} d_{k,t}x^k
$$
describing the weights of generators are determined as follows
\begin{align*}
P_1^{\w}=&\,x^4,\qquad\ \ 
P_2^{\w}=x^{-4}+x^{-2}+1,\qquad\ \ 
P_3^{\w}=x^{-8}+x^{-6}+x^{-4}+x^{-2}+1,\\
P_4^{\w}=&\,x^{-16}+x^{-14}+ x^{-12}+x^{-10} + 2 x^{-8}+x^{-6}+ x^{-4}+x^{-2}+1,\\
P_5^{\w}=&\,2 x^{-16}+{2}{x^{-14}}+{3}{x^{-12}}+{2}{x^{-10}}+{2}{x^{-8}}+x^{-6}+x^{-4}+x^{-2}+1,\\
P^{\w}_6 =&\, 2x^{-24 }+ 2x^{-22 }+ 3 x^{-20 }+ 3 x^{-18 }+3 x^{-16 }+ 3 x^{-14 }+ 3 x^{-12 } + 2 x^{-10 }+ 2 x^{-8 }+ x^{-6 }\\
&+ x^{-4 }+ x^{-2 }+ 1,\\
P^{\w}_7 =&\, x^{-26}+3x^{-24}+5x^{-22}+7x^{-20}+4x^{-18}+4x^{-16}+4x^{-14}+3x^{-12}+2x^{-10}+2x^{-8}\\
& +x^{-6}+x^{-4}+x^{-2}+1,\\
P^{\w}_8 =&\, 2x^{-32}+4x^{-30}+7x^{-28}+6x^{-26}+7x^{-24}+6x^{-22}+6x^{-20}+5x^{-18}+5x^{-16}+4x^{-14}\\
&+3x^{-12}+2x^{-10}+2x^{-8}+x^{-6}+x^{-4}+x^{-2}+1,\\
P^{\w}_9 =&\, x^{-36}+2x^{-34}+8x^{-32}+10x^{-30}+11x^{-28}+9x^{-26}+9x^{-24}+7x^{-22}+7x^{-20}+6x^{-18}\\
&+5x^{-16}+4x^{-14}+3x^{-12}+2x^{-10}+2x^{-8}+x^{-6}+x^{-4}+x^{-2}+1,\\
P^{\w}_{10} =&\, 4x^{-40}\!+7x^{-38}\!+11x^{-36}\!+12x^{-34}+14x^{-32}+12x^{-30}+12x^{-28}+11x^{-26}+10x^{-24}\!+8x^{-22}\\
&+8x^{-20}+6x^{-18}+5x^{-16}+4x^{-14}+3x^{-12}+2x^{-10}+2x^{-8}+x^{-6}+x^{-4}+x^{-2}+1,\\
P^{\w}_{11}=&\,5x^{-42}\!+15x^{-40}\!+19x^{-38}\!+20x^{-36}\!+16x^{-34}\!+17x^{-32}+15x^{-30}\!+14x^{-28}\!+12x^{-26}\!+11x^{-24}\\
&+9x^{-22}\!+8x^{-20}\!+6x^{-18}\!+5x^{-16}+4x^{-14}+3x^{-12}+2x^{-10}+2x^{-8}+x^{-6}+x^{-4}+x^{-2}+1,\\
P^{\w}_{12} =&\,{8}{x^{-48}}+{13}{x^{-46}}+{21}{x^{-44}}+{22}{x^{-42}}+{22}{x^{-40}}+{22}{x^{-38}}+{22}{x^{-36}}+{20}{x^{-34}}+{20}{x^{-32}}\\
&+{17}{x^{-30}}+{15}{x^{-28}}+{13}{x^{-26}}+{12}{x^{-24}}+9x^{-22}+8x^{-20}+6x^{-18}+5x^{-16}+4x^{-14}\\
&+3x^{-12}+2x^{-10}+2x^{-8}+x^{-6}+x^{-4}+x^{-2}+1,\\
P^{\w}_{13} =&\,{2}{x^{-52}}+{10}{x^{-50}}+{24}{x^{-48}}+{32}{x^{-46}}+{37}{x^{-44}}+{28}{x^{-42}}+{29}{x^{-40}}+{28}{x^{-38}}+{26}{x^{-36}}\\
    &+{23}{x^{-34}}+{22}{x^{-32}}+{18}{x^{-30}}+{16}{x^{-28}} + {14}{x^{-26}}+{12}{x^{-24}}+9x^{-22}+8x^{-20}+6x^{-18}\\
    &+5x^{-16}+4x^{-14}+3x^{-12}+2x^{-10}+2x^{-8}+x^{-6}+x^{-4}+x^{-2}+1.
\end{align*}
\end{theorem}

Clearly, the Laurent expansion of the following rational function at $x=0$ gives the dimension of the space of weak Jacobi forms of arbitrary weight and given index $t$
$$
\frac{P_t^{\w}}{(1-x^4)(1-x^6)} = \frac{\sum_{k\in \ZZ} d_{k,t}x^k}{(1-x^4)(1-x^6)} = \sum_{k\in \ZZ} \dim J_{k,E_8,t}^{\w, W(E_8)} x^k.
$$
This series is called the generating series of weak Jacobi forms of given index. We will present these generating series separately in Appendix I.

At the end of this section, we explicitly show some generators. For $W(E_8)$-invariant weak Jacobi form of index $t=2$, the three generators are
\begin{equation}
 \phi_{-4,2}=\frac{A_1^2 - A_2 E_4}{\Delta},\quad\  \phi_{-2,2}=\frac{A_2 E_6-B_2 E_4}{\Delta},\quad\  \phi_{0,2}=\frac{A_1^2 E_4 - B_2 E_6}{\Delta}.
\end{equation}
For index $t=3$, the five generators are constructed as
\begin{equation}
    \begin{aligned}
 \phi_{-8,3}=&\,\frac{1}{\Delta^2}(6 A_1^3 E_4 - 9 A_1 A_2 E_4^2 + A_3 (3E_4^3- 
 10E_6^2) + 30 A_1 B_2 E_6 - 20 B_3 E_4 E_6 ),\\
\phi_{-6,3}=&\,\frac{1}{\Delta^2}( 6 A_1^3 E_6 + 3 A_1 E_4 (10 B_2 E_4 - 3 A_2 E_6) - E_4^2 (20 B_3 E_4 + 7 A_3 E_6)),\\
\phi_{-4,3}=&\,\frac{1}{\Delta}( A_1 A_2 - A_3 E_4 ),\quad\ 
\phi_{-2,3}=\frac{1}{\Delta}(A_1 B_2 - A_3 E_6),\quad\ 
\phi_{0,3}=\frac{1}{\Delta}(A_1^3 - B_3 E_6 ).
    \end{aligned}
\end{equation}
For index $t=13$, the two lowest weight $-52$ generators are
\begin{equation}\nonumber
\begin{aligned}
\phi_{-52a,13}=&\,\frac{1}{\Delta^{\!10}E_4^2}(281154281472 A_1^{13} E_4^6 - 935841724416 A_1^{11} A_2 E_4^7-1468672041600 A_1^9 B_2^2 E_4^7+\dots)
\end{aligned}
\end{equation}
and $\phi_{-52b,13}=\phi_{-16,4}\phi_{-36,9}$, where
\begin{equation}
\begin{aligned}
\phi_{-16,4}=&\,\frac{1}{\Delta^{3}}(192 A_1^4 E_4 + 75 B_2^2 E_4^2 + 9 A_2^2 E_4^3 - A_4 E_4^4 + 90 A_2 B_2 E_4 E_6 + \dots)
,\\
\phi_{-36,9}=&\,\frac{1}{\Delta^{7}E_4}(1026432 A_1^9 E_4^4 - 2659392 A_1^7 A_2 E_4^5 - 5891400 A_1^5 B_2^2 E_4^5
+\dots),
\end{aligned}
\end{equation}
are the unique generators of the lowest weight for index $4$ and $9$ respectively.

\section{Weyl invariant holomorphic Jacobi forms}\label{sec:holomorphic}
Similar to weak Jacobi forms, the space of $W(E_8)$-invariant holomorphic Jacobi forms of integral weight and given index $t$
$$
J_{*,E_8,t}^{W(E_8)}:=\bigoplus_{k=4}^\infty J_{k,E_8,t}^{W(E_8)}
$$
is also a free module of rank $r(t)$ over $M_*(\SL_2(\ZZ))$. In this section we introduce some methods to compute the dimension of the space of holomorphic Jacobi forms. Firstly, the following lemma shows that the difference between the dimensions of the spaces of weak Jacobi forms and holomorphic Jacobi forms depends only on the index when the weight is greater than $4$.

\begin{proposition}\label{prop:delta}
For any $t\geq 1$ and any even $k\geq 6$, the following identity holds
\begin{equation}\label{eq:weak-holomorphic}
\dim J_{k,E_8,t}^{\w,W(E_8)} - \dim J_{k,E_8,t}^{W(E_8)} = \delta_t,
\end{equation}
where 
$$
\delta_t=\sum_{a=1}^\infty \epsilon_t(a) \cdot \delta_t(a), \quad \epsilon_t(a):=\min\left\{x\in \ZZ:  x\geq \frac{a}{t} \right\},
$$
and $\delta_t(a)$ is the number of elements of the set $\mathbb{S}_{t}(a)$ defined by
$$
\mathbb{S}_{t}(a)=\{  x=(x_i)_{i=1}^8 \in \NN^8 \setminus \{0\} : 2x_1 +3x_2 +4x_3 +6x_4 +5x_5 +4x_6 +3x_7 +2x_8\leq t, x^tSx=2a \},
$$
here $S$ is the Gram matrix associated to the fundamental weights $w_i$ of $E_8$ fixed in \cite[\S 3.1]{Wan21a}. The first values of $\delta_t$ are formulated in Table \ref{tab:delta_t}.
\end{proposition}

\begin{proof}
On the one hand, for any weak Jacobi form $\phi_t\in J_{k,E_8,t}^{\w,W(E_8)}$, if its Fourier expansion has no the following representatives of singular terms (i.e. $f(n,\ell)q^n\cdot \orb(\ell)$ satisfying $2n t-(\ell,\ell)<0$)
\begin{align*}
q^n \cdot \orb(x), \; 0\leq n < \epsilon_t(a), \; x \in \mathbb{S}_{t}(a),
\end{align*}
then $\phi_t$ is a holomorphic Jacobi form. Recall that $\orb(x)$ stands for the Weyl orbit of the vector $\sum_{i=1}^8 x_iw_i \in \Lambda_{+}$. This yields that $\dim J_{k,E_8,t}^{\w,W(E_8)} - \dim J_{k,E_8,t}^{W(E_8)} \leq  \delta_t$.

On the other hand, we can view Jacobi forms as vector-valued modular forms. By the theory of vector-valued modular forms for the Weil representation attached to the discriminant form of the rescaled lattice $E_8(t)$ (see \cite[Theorem 3.1]{Bor99} or \cite[Theorem 1.17]{Bru02}), the obstruction space, namely the space of cusp forms for the dual Weil representation, has weight $6-k$ and thus is trivial when $k\geq 6$. In the context of Jacobi forms, this implies that for each singular term above there exists a $W(E_8)$-invariant weak Jacobi form of weight $k$ and index $t$ whose Fourier expansion contains the given singular term but does not contain other singular terms. From this we conclude that $\dim J_{k,E_8,t}^{\w,W(E_8)} - \dim J_{k,E_8,t}^{W(E_8)} \geq  \delta_t$. We then prove the desired identity.
\end{proof}

\begin{table}[ht]
\caption{The value of $\delta_t$}\label{tab:delta_t}
\renewcommand\arraystretch{1.5}
\noindent\[
\begin{array}{|c|c|c|c|c|c|c|c|c|c|c|c|c|c|c|c|c|c|c|c|c}
\hline
t & 1& 2 & 3 & 4 & 5 & 6 & 7 &8&9&10 &11&12&13&14&15& 16 & 17 & 18\\ 
\hline 
\delta_t\! & 0&  2 & 5 & 13 & 23 & 52 & 82& 154&240& 403& 601& 959& 1373& 2063&2911 & 4184 &5739 & 8033\\ 
\hline 
\end{array} 
\]
\end{table}

By Proposition \ref{prop:delta} and the generating series of $J_{*,E_8,t}^{\w, W(E_8)}$, we can determine immediately the dimension of the space of $W(E_8)$-invariant holomorphic Jacobi forms of weight $k$ and index $t$ when $1\leq t \leq 13$ and $k\geq 6$. It remains to determine the space of holomorphic Jacobi forms of weight $4$. We see from the above proof that
\begin{equation}
    \dim J_{4,E_8,t}^{\w, W(E_8)} - \dim J_{4,E_8,t}^{W(E_8)} \leq  \delta_t. 
\end{equation}
The value of $\dim J_{4,E_8,t}^{W(E_8)}$ has been determined in \cite[Lemma 5.5]{Wan21a} when $t\leq 6$. By comparing the dimensions, we find that
\begin{equation}\label{eq:weight4}
    \dim J_{4,E_8,t}^{\w, W(E_8)} - \dim J_{4,E_8,t}^{W(E_8)} =  \delta_t, \quad \text{when $t\leq 6$}.
\end{equation}
However, the identity of type \eqref{eq:weight4} does not hold when $t\geq 7$. For example, $\dim J_{4,E_8,8}^{W(E_8)} \geq 2$, but $\dim J_{4,E_8,8}^{\w, W(E_8)} = \delta_8 =154$. 

As another application of Theorem \ref{th:Main}, we compute the dimension of the space of $W(E_8)$-invariant holomorphic Jacobi forms of weight $4$ and small index. These so-called holomorphic Jacobi forms of singular (i.e. possible minimal positive) weight are usually difficult to determine and construct in the theory of modular forms.

\begin{proposition}\label{prop:weight4}
The dimension of the space $J_{4,E_8,t}^{W(E_8)}$ for $t\leq 11$ is formulated in Table \ref{tab:singularforms}.
\begin{table}[ht]
\caption{The dimension of $J_{4,E_8,t}^{W(E_8)}$}\label{tab:singularforms}
\renewcommand\arraystretch{1.5}
\noindent\[
\begin{array}{|c|c|c|c|c|c|c|c|c|c|c|c|c|}
\hline
t & 1& 2 & 3 & 4 & 5 & 6 & 7 & 8 & 9 & 10 & 11   \\ 
\hline 
\mathrm{dim.} & 1&  1 & 1 & 2 & 1 & 1 & 2 & 2 & 2 & 2 & 2   \\ 
\hline 
\end{array} 
\]
\end{table}
\end{proposition}
\begin{proof}
There exist $W(E_8)$-invariant holomorphic Jacobi forms of weight $4$ and arbitrary index with Fourier expansion $1+O(q)$ (e.x. $X_t=1+O(q)$, see \eqref{eq:X_t}). Thus we only need to determine holomorphic Jacobi forms of weight $4$ and index $t$ whose $q^0$-term is zero. By Theorem \ref{th:Main}, such forms can be expressed as
$$
\frac{\sum_{j=0}^{t_1} P_j E_4^j P_{16, 5}^{t_1-j} }{\Delta^{N_t - 1}E_4^{t_1}}.
$$
Since these holomorphic forms have singular weight $4$, their Fourier expansion only involves Fourier coefficients $f(n,\ell)e^{2\pi i (n\tau + (\ell, \mathfrak{z}))}$ satisfying $(\ell,\ell)=2n t$ (see \cite{Gri94}). Moreover, these forms are completely determined  by coefficients of the following terms in their Fourier expansion:
$$
q^{\frac{1}{2t}(m,m)} \orb(m)
$$
where $m$ are non-zero vectors satisfying $(m,m)\in 2t\ZZ$ in the set (see \eqref{eq:T} for $T(m)$)
$$
\mathcal{A}_t:=\left\{ m=\sum_{i=1}^8 m_i w_i \in \Lambda_{+}: T(m)\leq t\right\}.
$$
We define $M_t$ as the smallest integer greater than or equal to the number
$$
\max\left\{ \frac{1}{2t}(m,m): m \in \mathcal{A}_t  \right\}.
$$
A $W(E_8)$-invariant holomorphic Jacobi form of weight $4$ and index $t$ whose $q^0$-term is zero corresponds to a solution of the system of linear equations defined by the Fourier expansion
\begin{equation}\label{eq:FE4}
    \frac{\sum_{j=0}^{t_1} P_j E_4^j P_{16, 5}^{t_1-j} }{\Delta^{N_t - 1}E_4^{t_1}} = \sum_{n=1}^{M_t -1} q^n \sum_{\substack{m \in \Lambda_{+}\\(m,m)=2nt }} \orb(m) + O(q^{M_t}).
\end{equation}
(By the proof of Proposition \ref{prop:delta}, there is no singular Fourier coefficient in $q^n$-term when $n\geq M_t$. Hence the above expression with given Fourier expansion must be a holomorphic Jacobi form.) We then prove the proposition by solving these systems of linear equations. When $t\geq 9$, our data on the Fourier coefficients of $A_i$ and $B_j$ is not sufficient to solve directly \eqref{eq:FE4}. In this case, we first solve \eqref{eq:FE4} up to $q^{10-N_t}$-terms. Fortunately, we find that all solutions have  an expression of type \eqref{eq:FE4}, replacing $N_t-1$ with a smaller power. We then prove the result for $t=9,10,11$ by solving an analogue of \eqref{eq:FE4}  with a smaller power of $\Delta$. 
\end{proof}

We give some direct constructions of generators of $J_{4,E_8,t}^{W(E_8)}$. When the space is one-dimensional, it is generated by the form $X_t=1+O(q)$ constructed in \eqref{eq:X_t}. When $t=4,8,9$, we construct the second generator as $A_4=A_1(\tau, 2\mathfrak{z})$, $A_2(\tau, 2\mathfrak{z})$ and $A_1(\tau, 3\mathfrak{z})$ respectively. We do not know how to construct the second generator of $J_{4,E_8,t}^{W(E_8)}$ in a direct way for $t=7,10,11$. However, they can be constructed  in terms of Sakai's forms. We present this nice construction in Appendix II.   Combining Proposition \ref{prop:delta} and Proposition \ref{prop:weight4}, it is easy to derive the generating series of $J_{*,E_8,t}^{W(E_8)}$ from the generating series of weak Jacobi forms. We omit these series here.

\section{Some conjectures}\label{sec: conj}
In this section we formulate some conjectures related to our work. The pull-back to $W(E_7)$-invariant Jacobi forms implies that the minimal weight of $W (E_8 )$-invariant weak Jacobi forms of index $t$ is not less than $-5t$ (see \cite[Proposition 6.1]{Wan21a}).  Here we propose a conjecture about the exact minimal weight. 

\begin{conjecture}\label{conj:weight}
The weight of non-zero $W (E_8 )$-invariant
weak Jacobi forms of index $t$ is not less than $-4t$.
\end{conjecture}
By our results in \S \ref{sec:weak}, there exist $W(E_8)$-invariant weak Jacobi forms of weight $-4t$ and index $t$ if $t$ is even and greater than $2$, or if $t=9$, or if $t$ is odd and greater than $11$. These forms can be constructed as monomials in our generators of index less than $10$. We give some evidence for this conjecture in the following lemma. 

\begin{lemma}
Let $t\geq 2$ and $\varphi_{k,t} \in J_{k,E_8,t}^{\w, W(E_8)} $. We assume that the coefficient of $\orb(\frac{t}{2} w_1)$ (resp. $\orb(\frac{t-3}{2}w_1+w_2)$) in the $q^0$-term of $\varphi_{k,t}$ is non-zero when $t$ is even (resp. odd). Then $k\geq -4t$ if $t$ is even, and $k\geq -4(t-1)$ if $t$ is odd.
\end{lemma}

\begin{proof}
We use the pull-back trick built in \cite[\S 5.6]{Wan21a}. Let $v_4$ be a vector of $E_8$ satisfying $(v_4,v_4)=4$. If $\varphi_{k,t} \in J_{k,E_8,t}^{\w, W(E_8)} $, then $\varphi_{k,t}(\tau, zv_4)\in J_{k,2t}^{\w}$ which is the space of weak Jacobi forms of weight $k$ and index $2t$ in the sense of Eichler and Zagier. Recall that the ring of classical weak Jacobi forms of integral index and even weight is freely generated by forms of index $1$ and weight $-2$ and $0$ denoted  $\phi_{-2,1}$ and $\phi_{0,1}$. We calculate 
$$
\max(\orb(x), v_4):=\max\{(y,v_4): y\in W(E_8)x\}=\max\{ (x, u_4): u_4 \in W(E_8)v_4 \}. 
$$
We find that
\begin{align*}
&\max(\orb(w_1),v_4)=4&  &\max(\orb(w_2),v_4)=5& &\max(\orb(w_3),v_4)=7& &\max(\orb(w_4),v_4)=10&\\
&\max(\orb(w_5),v_4)=8& &\max(\orb(w_6),v_4)=6& &\max(\orb(w_7),v_4)=4& &\max(\orb(w_8),v_4)=2.&
\end{align*}
Since $\max(\orb(x+y),v_4)=\max(\orb(x),v_4)+\max(\orb(y),v_4)$ and the average contributions of the eight fundamental Weyl orbits (with respect to index one) are respectively $4/2$, $5/3$, $7/4$, $10/6$, $8/5$, $6/4$, $4/3$, $2/2$, we conclude the following:
\begin{enumerate}
\item When $t=2m$, $\max(\orb(mw_1),v_4)=4m$ and $\max(\orb(x),v_4)<4m$ for all other $x\in \mathcal{A}_t$;

\item  When $t=2m+1$, $\max(\orb((m-1)w_1+w_2),v_4)=4m+1$ and $\max(\orb(x),v_4)<4m+1$ for all other $x\in \mathcal{A}_t$.
\end{enumerate} 
We note that the set $\mathcal{A}_t$ is defined in the proof of Proposition \ref{prop:weight4}.

When $t=2m$, the pull-back $\varphi_{k,t}(\tau, zv_4)=(\zeta^{\pm 4m} + \cdots) + O(q)$ is obviously non-zero, where $\zeta=e^{2\pi iz}$. If $k< -8m$, then $\varphi_{k,t}(\tau, zv_4) \in \phi_{-2,1}^{4m} \cdot J_{k+8m,0}^{\w}$, which leads to a contradiction because there is no non-zero $\SL_2(\ZZ)$ modular form of negative weight.

When $t=2m+1$, the pull-back $\varphi_{k,t}(\tau, zv_4)=(\zeta^{\pm (4m+1)} + \cdots) + O(q)$ is also non-zero. If $k< -8m$, then $\varphi_{k,t}(\tau, zv_4) \in \phi_{-2,1}^{4m} \cdot J_{k+8m,2}^{\w}$. In this case, $k+8m<0$. The only classical weak Jacobi forms of negative even weight and index $2$ are $\phi_{-2,1}^2$ and $\phi_{-2,1}\phi_{0,1}$, which all have leading Fourier coefficients $\zeta^{\pm 2}$ in their $q^0$-terms. This contradicts the $q^0$-term of $\varphi_{k,t}(\tau, zv_4)$. 
\end{proof}

We have checked that the $q^0$-term of the unique $W(E_8)$-invariant weak Jacobi form of weight $-36$ and index $9$ has no the Weyl orbit $\orb(3w_1+w_2)$, which is consistent with the above result.

Let $d_{k,t}$ be the number of weight $k$ generators of the free module $J_{*,E_8,t}^{\w,W(E_8)}$. By observing the generating series, we have the following conjecture on the stability.
\begin{conjecture}\label{conj:stability}
For any even $K\le 0$, there exist positive integer $L(K)$ such that for any fixed $k$ satisfying $K\le k\le 0$, the number $d_{k,t}$ is constant for all $t\ge L(K)$.
\end{conjecture}
We list the $L(K)$ and the relevant stable constants for weight $K\ge-24$ in Table \ref{tb:LK}.

\begin{table}[ht]
\caption{The values of $L(K)$ and $d_{K,L(K)}$.}\label{tb:LK}
\renewcommand\arraystretch{1.5}
\noindent\[
\begin{array}{|c|c|c|c|c|c|c|c|c|c|c|c|c|c|c|c|c|c|}
\hline
K & 0& -2 & -4 & -6 & -8 & -10 & -12 & -14 &-16 & -18 & -20 & -22& -24 \\ 
\hline 
L(K) & 2 & 2 & 2 & 3 & 4 & 5 &5 &7 &8 & 9 & 10 &11 &12  \\  \hline
d_{K,L(K)} & 1 & 1 & 1 & 1&  2 & 2 & 3 & 4 & 5 & 6 & 8 & 9 & 12\\
\hline
\end{array}
\]
\end{table}

For any irreducible root system not of type $E_8$, the ring of Weyl invariant weak Jacobi forms is a polynomial algebra. In such case, it is easy to prove the analogue of the above conjecture. 
We now prove the above conjecture for some special weights. 

\begin{proposition}
For any $t\geq 2$, the free module $J_{*,E_8,t}^{\w,W(E_8)}$ is generated by forms of non-positive weight. Moreover, the number of generators of weight $0$, $-2$ and $-4$ are all one. 
\end{proposition}
\begin{proof}
Let $t$ be a positive integer greater than $9$. 
We first fix some weak Jacobi forms. Let $\varphi_{-4,2}$, $\varphi_{-2,2}$ and $\varphi_{0,2}$ be the generators of $J_{*,E_8,2}^{\w,W(E_8)}$. Let $\varphi_{-8,3}$ and $\varphi_{-6,3}$ be the generators of weight $-8$ and $-6$ of $J_{*,E_8,3}^{\w,W(E_8)}$. Let $\varphi_{-16,4}$ and $\varphi_{-14,4}$ be the generators of weight $-16$ and $-14$ of $J_{*,E_8,4}^{\w,W(E_8)}$. Let $\varphi_{-16,5}$ be the generator of weight $-16$ of $J_{*,E_8,5}^{\w,W(E_8)}$ whose $q^0$-term involves the fundamental Weyl orbit $\orb(w_5)$. Let $\varphi_{-24,6}$ be the generator of weight $-24$ of $J_{*,E_8,6}^{\w,W(E_8)}$ whose $q^0$-term involves $\orb(w_4)$.  Since the $q^0$-terms of the eight negative-weight forms $\varphi_{-,-}$ involve the eight fundamental Weyl orbits respectively, the monomials of the eight forms in index $t$ have $q^0$-terms involving all monomials of fundamental Weyl orbits $\prod_{i=1}^8\orb(w_i)^{m_i}$ with $T(m)=t$. Thus they are linearly independent over $M_*(\SL_2(\ZZ))$. Moreover, their weights are not greater than $-t$. Similarly, the products of their monomials in index $t-1$ with $A_1$ have weight $\leq 5-t$ and have $q^0$-terms involving all monomials $\prod_{i=1}^8\orb(w_i)^{m_i}$ satisfying $T(m)=t-1$. We prove the proposition by induction on $t$. Suppose that $J_{*,E_8,t-2}^{\w,W(E_8)}$ is generated by forms of non-positive weight and has only one generator of weight $0$, $-2$ and $-4$ respectively.  Then the products of the generators of $J_{*,E_8,t-2}^{\w,W(E_8)}$ with $\varphi_{0,2}$ have $q^0$-terms involving all $\prod_{i=1}^8\orb(w_i)^{m_i}$ with $T(m)\leq t-2$. We have constructed $r(t)$ $W(E_8)$-invariant weak Jacobi forms of index $t$ which are linearly independent over $M_*(\SL_2(\ZZ))$. All of them have non-positive weight and the numbers of forms of weight $0$, $-2$, $-4$ are all one. Therefore, $J_{*,E_8,t}^{\w, W(E_8)}$ is generated by forms of non-positive weight and there is at most one generator of weight $k$ for $k=0,-2,-4$. It remains to prove that there do exist generators of weight $0,-2,-4$. The reduction of any weak Jacobi form of negative weight is identically zero when $\mathfrak{z}=0$. However, there are weak Jacobi forms of weight $0$ whose reduction is not zero.  Therefore, there are generators of weight $0$. If there is no generator of weight $-2$, then there are $r(t)$ weak Jacobi forms of weight $0$ whose $q^0$-terms are linearly independent. It follows that there exists a weak Jacobi form of weight $0$ whose $q^0$-term is non-zero constant, which contradicts \cite[Lemma 3.5]{Wan21a}. If there is no generator of weight $-4$, then there are $r(t)-1$ weak Jacobi forms of weight $-2$ whose $q^0$-terms are linearly independent. Hence there exists a weak Jacobi form of weight $-2$ whose $q^0$-term is $\orb(w_8)-240$. Acting the differential operator on this form (see \cite[Lemma 3.4]{Wan21a}), we can construct a weak Jacobi form of weight $0$ whose $q^0$-term is $(\frac{1}{2}-\frac{1}{t})\orb(w_8)-120$, which contradicts \cite[Lemma 3.5]{Wan21a} again. We have thus proved the desired result.
\end{proof}

In principle, the proof above should be able to extend to the cases of lower weights. 

It is known that the ring of weak Jacobi forms of integral weight and index one for a unimodular lattice $L$ is generated over $M_*(\SL_2(\ZZ))$ by the Jacobi theta function associated to $L$ which has positive weight $\frac{1}{2}\mathrm{rank}(L)$. As $tE_8$ lattice is no longer unimodular for $t\ge2$, inspired by the above result, we formulate a similar conjecture for general Jacobi forms on non-unimodular lattices. 

\begin{conjecture}
Let $L$ be an even positive definite lattice. Assume that $L$ is irreducible and is not unimodular. Then the free module of weak Jacobi forms of integral weight and index one associated to $L$ is generated by forms of non-positive weight. 
\end{conjecture}

We also make a conjecture on $W(E_8)$-invariant holomorphic Jacobi forms of singular weight. 

\begin{conjecture}
Let $H(t)$ be the dimension of the space of $W(E_8)$-invariant holomorphic Jacobi forms of weight $4$ and positive index $t$. Let $N(t)$ denote the number of distinct Weyl orbits of vectors of norm $t$ (i.e. $\frac{1}{2}(v,v)=t$). Then
$$
H(t)=N(t).
$$
Equivalently, for any Weyl orbit $\orb(m)$ of norm $t$, there exists a unique $W(E_8)$-invariant holomorphic Jacobi form of weight $4$ and index $t$ which has the Fourier expansion
$$
\Phi_{t,m}=1+\frac{240}{|\orb(m)|}q\cdot \orb(m)+O(q^2).
$$
\end{conjecture}

This conjecture has been proved for index $t\le 11$ in Proposition \ref{prop:weight4}. Because
we only calculated the Fourier expansions of $A_i$ and $B_j$ up to $q^9$-terms, it is not sufficient to extend Proposition \ref{prop:weight4} to index $12$ and $13$. However, by solving Jacobi forms of weight $4$ which can be expressed as $P(E_4,E_6,A_i,B_j) / \Delta^5E_4^2$ and have Fourier expansion of type \eqref{eq:FE4} up to $q^4$-terms, we find that the dimension of the solution space (including the form $X_t$) is $2$ when $t=12$ and $3$ when $t=13$.  These are consistent with the above conjecture. We formulate some values of $N(t)$ in Table \ref{tab:N(t)} below. 

\begin{table}[ht]
\caption{The value of $N(t)$}\label{tab:N(t)}
\renewcommand\arraystretch{1.5}
\noindent\[
\begin{array}{|c|c|c|c|c|c|c|c|c|c|c|c|c|c|c|c|c|c|c|c|c|c|c|c|}
\hline
 t & 1 & 2 & 3 & 4 & 5 & 6 & 7 & 8 & 9 & 10 & 11 & 12 & 13 & 14 & 15 & 16 & 17 & 18 & 19 & 20 & 21 & 22 & 23   \\
 \hline
 N(t) & 1 & 1 & 1 & 2 & 1 & 1 & 2 & 2 & 2 & 2 & 2 & 2 & 3 & 2 & 2 & 4 & 3 & 3 & 4 & 3  & 3 & 4 & 4\\
 \hline
\end{array}
\]
\end{table}

Finally, we formulate a conjecture on the global structure of $J_{*,E_8,*}^{\w,W(E_8)}$. This is motivated by Conjecture \ref{conj:weight} and Conjecture \ref{conj:stability}. In addition, it was proved in \cite{WW21} that the algebra of weak Jacobi forms for arbitrary rank-two lattice is finitely generated, which also motivates our conjecture.

\begin{conjecture}
The algebra of all $W(E_8)$-invariant weak Jacobi forms of integral weight and integral index is finitely generated over $M_*(\SL_2(\ZZ))$.
\end{conjecture}
By \S \ref{sec:weak}, the form $\phi_{-52a,13}$, which is one of the two generators of weight $-52$ for $J_{*,E_8,13}^{\w, W(E_8)}$, has to be a generator of the bigraded algebra $J_{*,E_8,*}^{\w,W(E_8)}$. This means that $J_{*,E_8,*}^{\w,W(E_8)}$ has generators of large index and its structure is extremely complicated.

\section*{Appendix I}
In this appendix we present the generating series of $J_{*,E_8,t}^{\w, W(E_8)}$ for index $t\leq 13$. Let us define
$$
\mathcal{J}_t^{\w}:=\frac{P_t^{\w}}{(1-x^4)(1-x^6)} = \frac{\sum_{k\in \ZZ} d_{k,t}x^k}{(1-x^4)(1-x^6)} = \sum_{k\in \ZZ} \dim J_{k,E_8,t}^{\w, W(E_8)} x^k.
$$
We here present the series $\mathcal{J}_t^{\w}$ up to $O(x^{22})$. It is known that the free module $J_{*,E_8,t}^{W(E_8)}$ of holomorphic Jacobi forms is generated by forms of weight not greater than $16$. Thus our data is sufficient to deduce the generating series of $J_{*,E_8,t}^{W(E_8)}$ from the results in \S \ref{sec:holomorphic}. 

\begin{align*}
\mathcal{J}_1^{\w}=&\,x^4+ x^8 + x^{10} + x^{12} + x^{14} + 2 x^{16} + x^{18} + 2 x^{20} + O(x^{22}). \\
\mathcal{J}_2^{\w}=&\, x^{-4}\! + x^{-2}\! + 2 + 2 x^2 + 3 x^4 + 3 x^6 + 4 x^8 + 4 x^{10}\!+ 5 x^{12}\! + 5 x^{14} + 6 x^{16}\! + 6 x^{18}\! + 7 x^{20}\! + O(x^{22}). \\
\mathcal{J}_3^{\w}=&\,x^{-8} + x^{-6} + 2 x^{-4} + 3 x^{-2} + 4 + 4 x^2 + 6 x^4 + 6 x^6 + 7 x^8 + 8 x^{10} + 9 x^{12} + 9 x^{14} + 11 x^{16}\\
&+ 11 x^{18} + 12 x^{20} + O(x^{22}). \\
\mathcal{J}_4^{\w}=&\,x^{-16} + x^{-14} + 2 x^{-12} + 3 x^{-10} + 5 x^{-8} + 5 x^{-6} + 8 x^{-4} + 9 x^{-2} + 11 + 12 x^2 + 15 x^4 + 15 x^6 \\
&+ 18 x^8 + 19 x^{10} + 21 x^{12} + 22 x^{14} + 25 x^{16} + 25 x^{18} + 28 x^{20} + O(x^{22}).\\
\mathcal{J}_5^{\w}=&\,{2}{x^{-16}}+{2}{x^{-14}}+{5}{x^{-12}}+{6}{x^{-10}}+{9}{x^{-8}}+{10}{x^{-6}}+{14}{x^{-4}}+{15}{x^{-2}}+19 +20 x^2 + 24 x^4\\
&+ 25 x^6 + 29 x^8 + 30 x^{10} + 34 x^{12} + 35 x^{14} + 39 x^{16} + 40 x^{18} + 44 x^{20} + O(x^{22}).\\
\mathcal{J}_6^{\w}=&\,{2}{x^{-24}}+{2}{x^{-22}}+{5}{x^{-20}}+{7}{x^{-18}}+{10}{x^{-16}}+{13}{x^{-14}}+{18}{x^{-12}}+{20}{x^{-10}} +{26}{x^{-8}} +{29}{x^{-6}}\\
&+{34}{x^{-4}}+{38}{x^{-2}}+44 + 46 x^2 + 53 x^4 + 56 x^6 + 61 x^8 + 65 x^{10 }+ 71 x^{12 }  + 73 x^{14} + 80 x^{16 }\\
&+ 83 x^{18} + 88 x^{20} +O(x^{22}).\\
\mathcal{J}_7^{\w}=&\,x^{-26}+{3}{x^{-24}}+{6}{x^{-22}}+{11}{x^{-20}}+{13}{x^{-18}}+{20}{x^{-16}}+{25}{x^{-14}}+{30}{x^{-12}}+{36}{x^{-10}}+{44}{x^{-8}}\\
& +{47}{x^{-6}}+{56}{x^{-4}}+{62}{x^{-2}}+68 +74 x^2+83 x^4+86 x^6+95 x^8+101 x^{10}+107 x^{12}+113 x^{14}\\
& +122 x^{16}+125 x^{18}+134 x^{20}+O(x^{22}).\\
\mathcal{J}_8^{\w}=&\,{2}{x^{-32}}\!+{4}{x^{-30}}+{9}{x^{-28}}+{12}{x^{-26}}+{20}{x^{-24}}+{25}{x^{-22}}+{34}{x^{-20}}\! +{41}{x^{-18}} \!+{52}{x^{-16}}\!+{59}{x^{-14}}\\
&+{71}{x^{-12}}+{79}{x^{-10}}+{91}{x^{-8}}  +{99}{x^{-6}} +{112}{x^{-4}} +{120}{x^{-2}}+133+141 x^2+154 x^4+162 x^6\\
&+175 x^8  +183 x^{10} +196 x^{12}+204 x^{14} +217 x^{16}+225 x^{18}+238 x^{20} +O(x^{22}).\\    
\mathcal{J}_9^{\w}=&\,{x^{-36}}+{2}{x^{-34}}+{9}{x^{-32}}+{13}{x^{-30}}+{22}{x^{-28}}+{30}{x^{-26}}+{42}{x^{-24}}+{50}{x^{-22}} +{66}{x^{-20}}\\
&+{76}{x^{-18}}+{91}{x^{-16}}+{104}{x^{-14}}+{120}{x^{-12}}+{131}{x^{-10}}+{150}{x^{-8}}+{161}{x^{-6}}+{178}{x^{-4}}\\
&+{192}{x^{-2}}+209+220 x^2+240 x^4+251 x^6+268 x^8 +282 x^{10}+299 x^{12}\! +310 x^{14}\!+330 x^{16}\\
&+341 x^{18}+358 x^{20} +O(x^{22}).\\
\mathcal{J}_{10}^{\w}=&\,{4}{x^{-40}}+{7}{x^{-38}}+{15}{x^{-36}}+{23}{x^{-34}}+{36}{x^{-32}}+{46}{x^{-30}}+{64}{x^{-28}}+{78}{x^{-26}}
+{97}{x^{-24}}\\
&+{114}{x^{-22}}+{137}{x^{-20}}+{153}{x^{-18}}+{178}{x^{-16}}+{197}{x^{-14}}+{220}{x^{-12}}+{240}{x^{-10}}+{266}{x^{-8}}\\
&+{283}{x^{-6}}+{310}{x^{-4}}+{330}{x^{-2}}+354+374 x^2+401 x^4 
+418 x^6+445 x^8+465 x^{10}\\
&+489 x^{12}+509 x^{14}+536 x^{16}+553 x^{18}+580 x^{20}+O(x^{22}).\\
\mathcal{J}_{11}^{\w}=&\,{5}{x^{-42}}+{15}{x^{-40}}+{24}{x^{-38}}+{40}{x^{-36}}+{55}{x^{-34}}+{76}{x^{-32}}+{95}{x^{-30}}+{121}{x^{-28}}+{143}{x^{-26}}\\
&+{172}{x^{-24}}+{197}{x^{-22}}+{228}{x^{-20}}+{254}{x^{-18}}+{287}{x^{-16}}+{314}{x^{-14}}+{347}{x^{-12}}+{375}{x^{-10}}\\
&+{409}{x^{-8}}+{436}{x^{-6}}+{471}{x^{-4}}+{499}{x^{-2}}+533+561 x^2 +596 x^4+623 x^6+658 x^8\\
&+686 x^{10}+720 x^{12}+748 x^{14}+783 x^{16}+810 x^{18}+845 x^{20}+O(x^{22}).\\
\mathcal{J}_{12}^{\w}=&\,{8}{x^{-48}}+{13}{x^{-46}}+{29}{x^{-44}}+{43}{x^{-42}}+{64}{x^{-40}}+{86}{x^{-38}}+{116}{x^{-36}}+{141}{x^{-34}}+{179}{x^{-32}}\\
&+{210}{x^{-30}}+{249}{x^{-28}}+{286}{x^{-26}}+{330}{x^{-24}}+{365}{x^{-22}}+{414}{x^{-20}}+{452}{x^{-18}}+{498}{x^{-16}}\\
&+{540}{x^{-14}}+{588}{x^{-12}}+{626}{x^{-10}}+{678}{x^{-8}}+{717}{x^{-6}}+{765}{x^{-4}}+{808}{x^{-2}}+857+895 x^2\\
&+948 x^4+987 x^6+1035 x^8+1078 x^{10}+1127 x^{12}+1165 x^{14}+1218 x^{16}+1257 x^{18}\\
&+1305 x^{20} + O(x^{22}).
\end{align*}

\begin{align*}
\mathcal{J}_{13}^{\w}=&\,  {2}{x^{-52}}+{10}{x^{-50}}+{26}{x^{-48}}+{44}{x^{-46}}+{73}{x^{-44}}+{96}{x^{-42}}+{136}{x^{-40}}+{171}{x^{-38}}+{214}{x^{-36}}\\
 &+{257}{x^{-34}}+{311}{x^{-32}}+{353}{x^{-30}}+{413}{x^{-28}}+{464}{x^{-26}}+{521}{x^{-24}}+{575}{x^{-22}}+{640}{x^{-20}}\\
 &+{689}{x^{-18}}+{756}{x^{-16}}+{812}{x^{-14}}+{873}{x^{-12}}+{930}{x^{-10}}+{998}{x^{-8}}+{1048}{x^{-6}}+{1117}{x^{-4}}\\
&+{1174}{x^{-2}}+1236+1293 x^2+1362 x^4+1412 x^6+1481 x^8+1538 x^{10}+1600 x^{12}+1657 x^{14}\\
&+1726 x^{16}+1776 x^{18}+1845 x^{20} + O(x^{22}).    
\end{align*}

\section*{Appendix II}
In this appendix we show the construction of the second generator $\Phi_t$ of $J_{*,E_8,t}^{W(E_8)}$ for $t=7, 10, 11$ (the form $X_t$ is the first generator). In the following, $P_{w,t}(E_4,E_6,A_i,B_j)$ are some very long polynomials of weight $w$ and index $t$ in $E_4$, $E_6$ and Sakai's nine forms $A_i$, $B_j$, of which we omit the explicit expression. The symbol $\cO_{a,d}^{[m_1...m_8]}$ stands for the Weyl orbit of the vector $\sum_{i=1}^8m_iw_i$ of norm $a$ whose number of elements is $d$. 

\begin{align*}
\Phi_7=&\,\frac{P_{32,7}(E_4,E_6,A_i,B_j)}{\Delta^2 E_4} \\
      =&\, 1+\frac{q}{56} \cO_{7,13440}^{[00000011]}+\frac{q^2}{280} \cO_{14,604800}^{[10000100]}+\frac{q^3}{280} \left(5 \cO_{21,13440}^{[00000013]}+\cO_{21,1814400}^{[10000101]}\right)\\
      &+\frac{q^4}{280} \left(5 \cO_{28,13440}^{[00000022]}+\cO_{28,4838400}^{[10001001]}\right)+O(q^5). \\
      \\
\Phi_{10}=&\,\frac{P_{56,10}(E_4,E_6,A_i,B_j)}{\Delta^4 E_4} \\
      =&\, 1+\frac{q}{126} \cO_{10,30240}^{[10000002]}+\frac{q^2}{630} \left( 5\cO_{20,30240}^{[20000002]} + \cO_{20,1209600}^{[10001000]} \right) +\frac{q^2}{1260} \left(2 \cO_{30,1814400}^{[10000102]}+\cO_{30,4838400}^{[00010010]}\right)\\
&+\frac{q^4}{1260} \left(10 \cO_{40,30240}^{[20000004]}+10\cO_{40,241920}^{[00002000]}+ 2 \cO_{40,4838400}^{[10100100]}+\cO_{40,9676800}^{[00010011]}\right)\\
&+\frac{q^5}{1260} \Big(140 \cO_{50,2160}^{[50000000]}+10\cO_{50,30240}^{[10000006]}+ 2 \cO_{50,1814400}^{[10000120]}+2\cO_{50,4838400}^{[10001003]}\\
&\phantom{----}+ 2 \cO_{50,4838400}^{[10101000]}+\cO_{50,14515200}^{[00010101]}\Big)+O(q^6).\\
\\
\Phi_{11}=&\,\frac{P_{60,11}(E_4,E_6,A_i,B_j)}{\Delta^4 E_4^2} \\
=&\,1+\frac{q}{756}  \cO _{11,181440}^{[ {00000101]}}+\frac{q^2}{3780} \left(3 \cO _{22,2419200}^{[ {00010001]}}+5 \cO _{22,181440}^{[ {10000020]}}\right)\\
&+\frac{q^3}{7560} \left(10 \cO _{33,181440}^{[ {00000201]}}+6 \cO _{33,2419200}^{[ {00101000]}}+10 \cO _{33,362880}^{[ {10000013]}}+5 \cO _{33,5806080}^{[ {11000011]}}+10 \cO _{33,181440}^{[ {30000010]}}\right)\\
&+\frac{q^4 }{7560}\left(10 \cO _{44,181440}^{[ {00000202]}}+6 \cO _{44,7257600}^{[ {00100110]}}+6 \cO _{44,4838400}^{[ {01010001]}}+5 \cO _{44,5806080}^{[ {11000012]}}+6 \cO _{44,4838400}^{[ {20001001]}}\right)\\
&+\frac{q^5}{7560} \Big(10 \cO _{55,362880}^{[ {00000211]}}+6 \cO _{55,2419200}^{[ {00010004]}}+6 \cO _{55,2419200}^{[ {00100200]}}+6 \cO _{55,9676800}^{[ {00101002]}}+6 \cO _{55,9676800}^{[ {01010010]}}\\
&\phantom{----}+10 \cO _{55,362880}^{[ {10000023]}}+5 \cO _{55,5806080}^{[ {11000021]}}+6 \cO _{55,7257600}^{[ {20001010]}}+10 \cO _{55,362880}^{[ {30000012]}}\Big)+O(q^6).
\end{align*}
We remark that there are in total $11$ Weyl orbits of norm $55$ and nine of them are involved in the $q^5$-term of $\Phi_{11}$. As explained in the proof of Proposition \ref{prop:weight4}, the Fourier expansion of any $W(E_8)$-invariant holomorphic Jacobi form of singular weight $4$ and index $t$ is completely determined by the finitely-many Weyl orbits related to the action of $W(E_8)$ on the discriminant group $E_8/tE_8$. Therefore, in each Fourier expansion of $\Phi_t$ or $X_t$ there are only a finite number of different coefficients for the Weyl orbits. 

\bigskip

\noindent
\textbf{Acknowledgements} 
This work began in July 2020, when both authors were postdoctoral fellows at the Max Planck Institute for Mathematics in Bonn. The authors thank the institute for its stimulating environment and financial support. The authors would like to thank Valery Gritsenko, Min-xin Huang, Albrecht Klemm, Kimyeong Lee,
Kazuhiro Sakai, Xin Wang for useful discussions. KS is also supported by Korea Institute for Advanced Study Grant QP081001. H. Wang is supported by the Institute for Basic Science (IBS-R003-D1). 

\bibliographystyle{plainnat}
\bibliofont
\bibliography{refs}

\end{document}